\documentclass{amsart}
\usepackage[utf8]{inputenc}
\usepackage{xcolor}
\usepackage{amsmath}
\usepackage{amssymb}
\usepackage{bbm}
\usepackage{tikz}

\usepackage{amsthm}
\newtheorem{thm}{Theorem}[section]
\newtheorem{prop}[thm]{Proposition}
\newtheorem{lemma}[thm]{Lemma}
\newtheorem{cor}[thm]{Corollary}

\usepackage{geometry}
\newcommand{\N}{\ensuremath{\mathbb{N}}}
\newcommand{\Sn}{\ensuremath{\mathcal{S}_n}}
\newcommand{\rec}{\ensuremath{\mathrm{record}}}
\newcommand{\Root}{\ensuremath{\varnothing}}
\newcommand{\lChild}{\ensuremath{\overline{0}}}
\newcommand{\rChild}{\ensuremath{\overline{1}}}
\newcommand{\bst}[1]{\ensuremath{T\langle{#1}\rangle}}
\newcommand{\bstl}[1]{\ensuremath{\tau\langle{#1}\rangle}}

\newcommand{\undersigma}{\ensuremath{\sigma_-}}
\newcommand{\oversigma}{\ensuremath{\sigma_+}}

\definecolor{theme}{RGB}{15, 87, 24} 
\definecolor{lighttheme}{RGB}{51, 153, 63} 
\definecolor{block}{RGB}{102, 60, 0} 
\definecolor{lightblock}{RGB}{255, 238, 214} 
\definecolor{alert}{RGB}{212, 43, 43} 

\usepackage{hyperref}
\hypersetup{colorlinks=true,
            linkcolor=theme,
            citecolor=lighttheme}

\title{The height of record-biased trees}
\author{
    Beno{\^i}t Corsini
}
\date{December 9, 2021}

\begin{document}

\maketitle

\begin{abstract}
    Given a permutation $\sigma$, its corresponding binary search tree is obtained by recursively inserting the values $\sigma(1),\ldots,\sigma(n)$ into a binary tree so that the label of each node is larger than the labels of its left subtree and smaller than the labels of its right subtree. In 1986, Devroye proved that the height of such trees when $\sigma$ is a random uniform permutation is of order $(c^*+o_\mathbb{P}(1))\log n$ as $n$ tends to infinity, where $c^*$ is the only solution to $c\log(2e/c)=1$ with $c\geq2$. In this paper, we study the height of binary search trees drawn from the record-biased model of permutations, introduced by Auger, Bouvel, Nicaud, and Pivoteau in 2016. The record-biased distribution is the probability measure on the set of permutations whose weight is proportional to $\theta^{\rec(\sigma)}$, where $\rec(\sigma)=|\{i\in[n]:\forall j<i,\sigma(i)>\sigma(j)\}|$. We show that the height of a binary search tree built from a record-biased permutation of size $n$ with parameter $\theta$ is of order $(1+o_\mathbb{P}(1))\max\{c^*\log n,\,\theta\log(1+n/\theta)\}$, hence giving a full characterization of the first order asymptotic behaviour of the height of such trees.
\end{abstract}

\tableofcontents

\section{Introduction}

Let $\N=\{0,1,2,\ldots\}$ be the set of non-negative integers. For $n\in\N$, write $\Sn$ for the set of permutations of the set $[n]=\{1,\ldots,n\}$. In this article, we study properties of \textit{record-biased permutations}~\cite{auger2016analysis}, defined as follows. Let $n\in\N$ and $\theta\in[0,\infty)$. The record-biased distribution with parameters $n$ and $\theta$ is the probability measure $w_{n,\theta}$ on $\Sn$ defined by
\begin{align*}
    w_{n,\theta}(\sigma)=\frac{\theta^{\rec(\sigma)}}{W_{n,\theta}}\,,
\end{align*}
where $\rec(\sigma)=|\{i\in[n]:\forall j\in[i-1],\,\sigma(i)>\sigma(j)\}|$ is the number of records of $\sigma$ and $W_{n,\theta}=\sum_{\sigma\in\Sn}\theta^{\rec(\sigma)}$ is a normalizing constant. This model of random permutations was introduced by Auger, Bouvel, Nicaud, and Pivoteau~\cite{auger2016analysis}, who characterized some of its properties: the asymptotic distribution of the number of records, the number of descents, the first value $\sigma(1)$, and the number of inversions. This paper derives the first order asymptotic behaviour of the heights of the binary search trees associated to record-biased permutations.

\medskip

Let $T_\infty=\{\Root\}\cup\bigcup_{k\geq1}\{\lChild,\rChild\}^k$ be the complete infinite rooted binary tree, where nodes $u$ at depth $|u|\geq1$ are indexed by strings written as $u=u_1\ldots u_{|u|}\in\{\lChild,\rChild\}^{|u|}$. This means that $u$ has parent $u_1\ldots u_{|u|-1}$ and children $u\lChild$ and $u\rChild$. For a set $V\subset T_\infty$ and a node $u\in T_\infty$, let $uV=\{uv,v\in V\}$.

Call a \textit{subtree} of $T_\infty$ (or just ``tree'', for short) a set $T\subset T_\infty$ which is connected when viewed as a subgraph of $T_\infty$. For any subtree $T$ of $T_\infty$, its \textit{root} is defined to be the unique node of $T$ of minimum depth. Given a tree $T$ and any node $u\in T_{\infty}$, let $T(u)=(uT_\infty)\cap T$ be the subtree of $T$ rooted at $u$; note that if $\Root\in T$, then $T(u)=\emptyset$ if and only if $u\notin T$. Finally, for any subtree $T\subset T_\infty$, write $h(T)=\sup(|u|,u\in T)-\inf(|u|,u\in T)$ for its height, corresponding to the greatest distance between any node and the root of $T$. 

Call a \textit{labelled tree} any pair $(T,\tau)$ where $T$ is a subtree of $T_\infty$ and $\tau:T\rightarrow\N$ is an injective function. Furthermore say that $(T,\tau)$ is a \textit{binary search tree} if, for any $u\in T$, and for any $v\in T(u\lChild)$ (respectively $v\in T(u\rChild)$), we have $\tau(u)>\tau(v)$ (respectively $\tau(u)<\tau(v)$); in other words, a binary search tree is a labelled tree where the label of each node is larger than the labels of its left subtree and smaller than the labels of its right subtree.

Given an injective function $f:[n]\rightarrow\N$, call \textit{binary search tree of $f$} and write $\big(\bst{f},\bstl{f}\big)$ for the unique binary search tree such that, for all $i\in[n]$, the set
\begin{align*}
    \bstl{f}^{-1}\Big(\big\{f(1),\ldots,f(i)\big\}\Big)
\end{align*}
is a subtree of $T_\infty$ of size $i$ and rooted at $\Root$. This definition can also be rephrased inductively. Indeed, $\big(\bst{f|_{[i]}},\bstl{f|_{[i]}}\big)$ is obtained from $\big(\bst{f|_{[i-1]}},\bstl{f|_{[i-1]}}\big)$ by inserting the value $f(i)$ in the labelled tree $\big(\bst{f|_{[i-1]}},\bstl{f|_{[i-1]}}\big)$ at the unique location for which $\big(\bst{f|_{[i]}},\bstl{f|_{[i]}}\big)$ remains a binary search tree. An example of this construction can be found in Figure~\ref{fig:BST}.

\begin{figure}[htb]
    \centering
    \begin{tikzpicture}[scale=0.4]
        \begin{scope}
            \draw[line width=0.02cm, lightblock!80!block] (0,0) grid (7,-7);
            \draw[line width=0.025cm, lightblock!30!block, ->] (0,0) -- (7.5,0);
            \draw[line width=0.025cm, lightblock!30!block, ->] (0,0) -- (0,-7.5);
            \node[lightblock!30!block, anchor=east, scale=0.8] at (0,-7){$i$};
            \node[lightblock!30!block, anchor=south, scale=0.8] at (7,0){$f(i)$};
            \node[draw, circle, theme, fill=lighttheme, line width=0.05cm, scale=1.3] at (2,-1){};
            \node[white, scale=0.9](2) at (2,-1){\textbf{2}};
            \node[anchor=north west, scale=0.8] at (0.2,-7.1){$f=(2)$};
        \end{scope}
        \begin{scope}[xshift=10cm]
            \draw[line width=0.02cm, lightblock!80!block] (0,0) grid (7,-7);
            \draw[line width=0.025cm, lightblock!30!block, ->] (0,0) -- (7.5,0);
            \draw[line width=0.025cm, lightblock!30!block, ->] (0,0) -- (0,-7.5);
            \node[lightblock!30!block, anchor=east, scale=0.8] at (0,-7){$i$};
            \node[lightblock!30!block, anchor=south, scale=0.8] at (7,0){$f(i)$};
            \node(2) at (2,-1){};
            \node(4) at (4,-2){};
            \draw[theme, line width=0.05cm] (2.center) -- (4.center);
            \node[draw, circle, theme, fill=white, line width=0.05cm, scale=1.3] at (2,-1){};
            \node[draw, circle, theme, fill=lighttheme, line width=0.05cm, scale=1.3] at (4,-2){};
            \node[theme, scale=0.9](2) at (2,-1){\textbf{2}};
            \node[white, scale=0.9](4) at (4,-2){\textbf{4}};
            \node[anchor=north west, scale=0.8] at (0.2,-7.1){$f=(2,4)$};
        \end{scope}
        \begin{scope}[xshift=20cm]
            \draw[line width=0.02cm, lightblock!80!block] (0,0) grid (7,-7);
            \draw[line width=0.025cm, lightblock!30!block, ->] (0,0) -- (7.5,0);
            \draw[line width=0.025cm, lightblock!30!block, ->] (0,0) -- (0,-7.5);
            \node[lightblock!30!block, anchor=east, scale=0.8] at (0,-7){$i$};
            \node[lightblock!30!block, anchor=south, scale=0.8] at (7,0){$f(i)$};
            \node(2) at (2,-1){};
            \node(4) at (4,-2){};
            \node(1) at (1,-3){};
            \draw[theme, line width=0.05cm] (2.center) -- (1.center);
            \draw[theme, line width=0.05cm] (2.center) -- (4.center);
            \node[draw, circle, theme, fill=white, line width=0.05cm, scale=1.3] at (2,-1){};
            \node[draw, circle, theme, fill=white, line width=0.05cm, scale=1.3] at (4,-2){};
            \node[draw, circle, theme, fill=lighttheme, line width=0.05cm, scale=1.3] at (1,-3){};
            \node[theme, scale=0.9](2) at (2,-1){\textbf{2}};
            \node[theme, scale=0.9](4) at (4,-2){\textbf{4}};
            \node[white, scale=0.9](1) at (1,-3){\textbf{1}};
            \node[anchor=north west, scale=0.8] at (0.2,-7.1){$f=(2,4,1)$};
        \end{scope}
        \begin{scope}[yshift=-10cm]
            \draw[line width=0.02cm, lightblock!80!block] (0,0) grid (7,-7);
            \draw[line width=0.025cm, lightblock!30!block, ->] (0,0) -- (7.5,0);
            \draw[line width=0.025cm, lightblock!30!block, ->] (0,0) -- (0,-7.5);
            \node[lightblock!30!block, anchor=east, scale=0.8] at (0,-7){$i$};
            \node[lightblock!30!block, anchor=south, scale=0.8] at (7,0){$f(i)$};
            \node(2) at (2,-1){};
            \node(4) at (4,-2){};
            \node(1) at (1,-3){};
            \node(6) at (6,-4){};
            \draw[theme, line width=0.05cm] (2.center) -- (1.center);
            \draw[theme, line width=0.05cm] (2.center) -- (4.center);
            \draw[theme, line width=0.05cm] (4.center) -- (6.center);
            \node[draw, circle, theme, fill=white, line width=0.05cm, scale=1.3] at (2,-1){};
            \node[draw, circle, theme, fill=white, line width=0.05cm, scale=1.3] at (4,-2){};
            \node[draw, circle, theme, fill=white, line width=0.05cm, scale=1.3] at (1,-3){};
            \node[draw, circle, theme, fill=lighttheme, line width=0.05cm, scale=1.3] at (6,-4){};
            \node[theme, scale=0.9](2) at (2,-1){\textbf{2}};
            \node[theme, scale=0.9](4) at (4,-2){\textbf{4}};
            \node[theme, scale=0.9](1) at (1,-3){\textbf{1}};
            \node[white, scale=0.9](6) at (6,-4){\textbf{6}};
            \node[anchor=north west, scale=0.8] at (0.2,-7.1){$f=(2,4,1,6)$};
        \end{scope}
        \begin{scope}[xshift=10cm,yshift=-10cm]
            \draw[line width=0.02cm, lightblock!80!block] (0,0) grid (7,-7);
            \draw[line width=0.025cm, lightblock!30!block, ->] (0,0) -- (7.5,0);
            \draw[line width=0.025cm, lightblock!30!block, ->] (0,0) -- (0,-7.5);
            \node[lightblock!30!block, anchor=east, scale=0.8] at (0,-7){$i$};
            \node[lightblock!30!block, anchor=south, scale=0.8] at (7,0){$f(i)$};
            \node(2) at (2,-1){};
            \node(4) at (4,-2){};
            \node(1) at (1,-3){};
            \node(6) at (6,-4){};
            \node(3) at (3,-5){};
            \draw[theme, line width=0.05cm] (2.center) -- (1.center);
            \draw[theme, line width=0.05cm] (2.center) -- (4.center);
            \draw[theme, line width=0.05cm] (4.center) -- (3.center);
            \draw[theme, line width=0.05cm] (4.center) -- (6.center);
            \node[draw, circle, theme, fill=white, line width=0.05cm, scale=1.3] at (2,-1){};
            \node[draw, circle, theme, fill=white, line width=0.05cm, scale=1.3] at (4,-2){};
            \node[draw, circle, theme, fill=white, line width=0.05cm, scale=1.3] at (1,-3){};
            \node[draw, circle, theme, fill=white, line width=0.05cm, scale=1.3] at (6,-4){};
            \node[draw, circle, theme, fill=lighttheme, line width=0.05cm, scale=1.3] at (3,-5){};
            \node[theme, scale=0.9](2) at (2,-1){\textbf{2}};
            \node[theme, scale=0.9](4) at (4,-2){\textbf{4}};
            \node[theme, scale=0.9](1) at (1,-3){\textbf{1}};
            \node[theme, scale=0.9](6) at (6,-4){\textbf{6}};
            \node[white, scale=0.9](3) at (3,-5){\textbf{3}};
            \node[anchor=north west, scale=0.8] at (0.2,-7.1){$f=(2,4,1,6,3)$};
        \end{scope}
        \begin{scope}[xshift=20cm,yshift=-10cm]
            \draw[line width=0.02cm, lightblock!80!block] (0,0) grid (7,-7);
            \draw[line width=0.025cm, lightblock!30!block, ->] (0,0) -- (7.5,0);
            \draw[line width=0.025cm, lightblock!30!block, ->] (0,0) -- (0,-7.5);
            \node[lightblock!30!block, anchor=east, scale=0.8] at (0,-7){$i$};
            \node[lightblock!30!block, anchor=south, scale=0.8] at (7,0){$f(i)$};
            \node(2) at (2,-1){};
            \node(4) at (4,-2){};
            \node(1) at (1,-3){};
            \node(6) at (6,-4){};
            \node(3) at (3,-5){};
            \node(5) at (5,-6){};
            \draw[theme, line width=0.05cm] (2.center) -- (1.center);
            \draw[theme, line width=0.05cm] (2.center) -- (4.center);
            \draw[theme, line width=0.05cm] (4.center) -- (3.center);
            \draw[theme, line width=0.05cm] (4.center) -- (6.center);
            \draw[theme, line width=0.05cm] (6.center) -- (5.center);
            \node[draw, circle, theme, fill=white, line width=0.05cm, scale=1.3] at (2,-1){};
            \node[draw, circle, theme, fill=white, line width=0.05cm, scale=1.3] at (4,-2){};
            \node[draw, circle, theme, fill=white, line width=0.05cm, scale=1.3] at (1,-3){};
            \node[draw, circle, theme, fill=white, line width=0.05cm, scale=1.3] at (6,-4){};
            \node[draw, circle, theme, fill=white, line width=0.05cm, scale=1.3] at (3,-5){};
            \node[draw, circle, theme, fill=lighttheme, line width=0.05cm, scale=1.3] at (5,-6){};
            \node[theme, scale=0.9](2) at (2,-1){\textbf{2}};
            \node[theme, scale=0.9](4) at (4,-2){\textbf{4}};
            \node[theme, scale=0.9](1) at (1,-3){\textbf{1}};
            \node[theme, scale=0.9](6) at (6,-4){\textbf{6}};
            \node[theme, scale=0.9](3) at (3,-5){\textbf{3}};
            \node[white, scale=0.9](5) at (5,-6){\textbf{5}};
            \node[anchor=north west, scale=0.8] at (0.2,-7.1){$f=(2,4,1,6,3,5)$};
        \end{scope}
    \end{tikzpicture}
    \caption{An example of construction of the binary search tree $(\bst{\sigma},\bstl{\sigma})$ for the permutation $\sigma=(2,4,1,6,3,5)$. The nodes highlighted in green correspond to the most recently inserted elements, and the values on the nodes correspond to the values of the labelling function $\bstl{\sigma}$. This figure represents the recursive construction of binary search trees, where values are inserted one after the other at the unique location such that the label of each node is larger than the labels of its left subtree and smaller than the labels of its right subtree.}
    \label{fig:BST}
\end{figure}
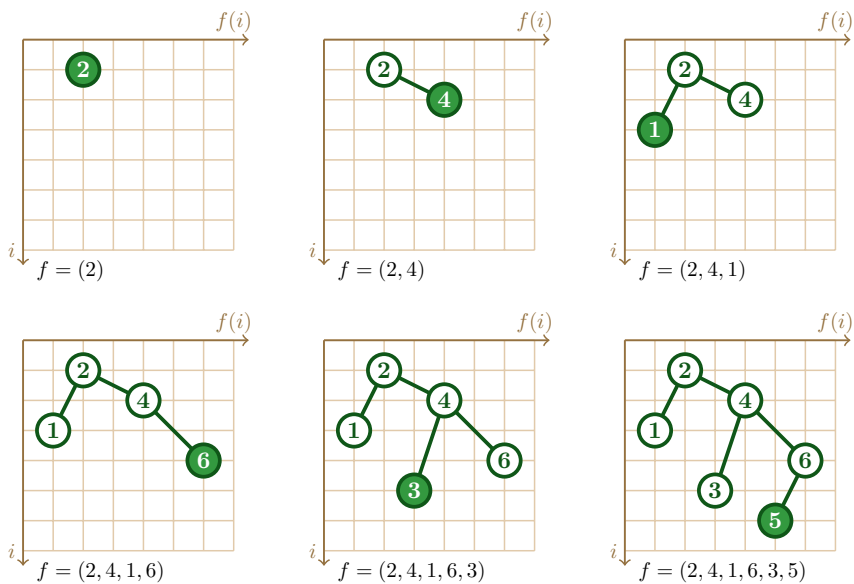

For the rest of the paper, write $T_{n,\theta}$ for a \textit{record-biased tree}, defined to be distributed as $\bst{\sigma}$ where $\sigma$ is $w_{n,\theta}$-distributed. Let $c^*=4.311\ldots$ be the unique solution of $c\log\left(\frac{2e}{c}\right)=1$ with $c\geq2$. The goal of this work is to prove the following result on the height of record-biased trees.

\begin{thm}\label{thm:combined}
    Let $(\theta_n)_{n\geq0}$ be any sequence of non-negative numbers. Then, as $n$ tends to infinity,
    \begin{align*}
        \frac{h(T_{n,\theta_n})}{\max\left\{c^*\log n,\,\theta_n\log\left(1+\frac{n}{\theta_n}\right)\right\}}\longrightarrow1\,,
    \end{align*}
    in probability and in $L^p$ for any $p>0$.
\end{thm}

For any $n\in\N$ and $\theta\in[0,\infty)$, write
\begin{align}
    \mu(n,\theta)=\sum_{0\leq i<n}\frac{\theta}{\theta+i}\,.\label{eq:mu}
\end{align}
By comparison to the integral, we can see that
\begin{align*}
    \bigg|\mu(n,\theta)-\left[1+\theta\log\left(1+\frac{n}{\theta}\right)\right]\bigg|\leq\theta\log\left(1+\frac{1}{\theta}\right)\,,
\end{align*}
which implies that, as $n$ tends to infinity, we have $\frac{\mu(n,\theta_n)}{1+\theta_n\log(1+n/\theta_n)}\rightarrow1$. It follows that, for any sequence $(\theta_n)_{n\geq0}$ of non-negative numbers,
\begin{align*}
    \max\Big\{c^*\log n,\,\mu(n,\theta_n)\Big\}=\big(1+o(1)\big)\max\left\{c^*\log n,\,\theta_n\log\left(1+\frac{n}{\theta_n}\right)\right\}\,.
\end{align*}
This implies that the convergence in Theorem~\ref{thm:combined} is equivalent to the statement that
\begin{align*}
    \frac{h(T_{n,\theta_n})}{\max\big\{c^*\log n,\,\mu(n,\theta_n)\big\}}\longrightarrow1
\end{align*}
in probability and in $L^p$ for any $p>0$. For the rest of the paper, we aim at proving this statement rather than the one of Theorem~\ref{thm:combined}.

By taking subsequences, to prove Theorem~\ref{thm:combined}, it suffices to consider two cases: if $\theta_n\equiv\theta\in[0,\infty)$ is constant, or if $\theta_n\rightarrow\infty$ as $n\rightarrow\infty$. In the case where $\theta_n=\theta$ for all $n\geq0$, then $\mu(n,\theta_n)\sim\theta\log n$. This case of Theorem~\ref{thm:combined} can be rewritten as follows.

\begin{thm}\label{thm:height}
    For any non-negative number $\theta\in[0,\infty)$, as $n$ tends to infinity,
    \begin{align*}
        \frac{h(T_{n,\theta})}{\log n}\longrightarrow\max\big\{c^*,\theta\big\}\,
    \end{align*}
    in probability and in $L^p$ for any $p>0$.
\end{thm}

When $\theta=1$, the tree $T_{n,\theta}$ is a \textit{random binary search tree}, the binary search tree of a uniformly random permutation, and this case of Theorem~\ref{thm:height} corresponds to a well-known result of Devroye~\cite{devroye1986note}; this case will also be used as input for the proof of the general result. Theorem~\ref{thm:height} furthermore shows that there is a change of behaviour for the height of record-biased trees at the value $\theta=c^*$. Since previous results~\cite{auger2016analysis} on record-biased permutations proved that the number of records is of order $\theta\log n$, this change of behaviour corresponds to the moment where the first-order height of the tree becomes characterized by the length of its rightmost path.

Note that, Theorem~\ref{thm:height} also states that the asymptotic behaviour of the height does not change around the value $\theta=1$. This result might be unexpected since, for $\theta<1$ the records of the permutations are penalized, whereas they are rewarded when $\theta>1$. This strong change of behaviour for the permutation does not affect the height of the tree since, in the case where $\theta<c^*$, the height of $T_{n,\theta}$ is actually mainly characterized by the height of the left subtree of the root.

Theorem~\ref{thm:height} covers the case where $(\theta_n)_{n\geq0}$ is constant in Theorem~\ref{thm:combined}. For the case where $(\theta_n)_{n\geq0}$ diverges to infinity, first note that, for such sequences, $\mu(n,\theta_n)=\omega(\log n)$, where $a_n=\omega(b_n)$ means that $|a_n/b_n|\rightarrow\infty$. This case of Theorem~\ref{thm:combined} can thus be rewritten (and strengthened) as follows.

\begin{thm}\label{thm:strongHeight}
    Let $(\theta_n)_{n\geq0}$ be a sequence of non-negative numbers such that $\theta_n\rightarrow\infty$. Then, as $n$ tends to infinity,
    \begin{align*}
        \frac{h(T_{n,\theta_n})}{\mu(n,\theta_n)}\longrightarrow1
    \end{align*}
    in probability and in $L^p$ for any $p>0$. Moreover, for any $\epsilon>0$ and $\lambda>0$,
    \begin{align*}
        \mathbb{P}\left(\left|\frac{h(T_{n,\theta_n})}{\mu(n,\theta_n)}-1\right|>\epsilon\right)=O\left(\frac{1}{n^\lambda}\right)\,.
    \end{align*}
\end{thm}

By using the first Borel-Cantelli lemma, the second bound of this result implies that, whenever the random variables of the sequence $(T_{n,\theta_n})_{n\geq0}$ are defined on a common probability space, the first convergence also occurs almost-surely. Together, Theorems~\ref{thm:height} and \ref{thm:strongHeight} establish Theorem~\ref{thm:combined}.

\subsection{Overview of the proofs}\label{subsec:overview}

The main strategy to prove Theorems~\ref{thm:height} and \ref{thm:strongHeight} is based on a similar method to that of~\cite{addario2021height}: controlling the size of the rightmost path and the behaviour of the left subtrees hanging from this rightmost path. We now provide an overview of the proof of Theorem~\ref{thm:height} - a similar method is used to prove Theorem~\ref{thm:strongHeight}. In particular Proposition~\ref{prop:inductiveTree}, Lemma~\ref{lem:boundsRecords}, and Proposition~\ref{prop:upperBoundL} below are all non-asymptotic results, so can be applied in the setting of either theorem.

The main result that we use to understand the overall structure of the tree is given by the following proposition.

\begin{prop}\label{prop:inductiveTree}
    Let $n\in\N$ and $\theta\in[0,\infty)$. Let $T_{n,\theta}$ be a record-biased tree with parameters $n$ and $\theta$. Then, for any $k\in[n]$,
    \begin{align*}
        \mathbb{P}\Big(\big|T_{n,\theta}(\lChild)\big|=k-1\Big)=\mathbb{P}\Big(\big|T_{n,\theta}(\rChild)\big|=n-k\Big)=\frac{\theta}{\theta+n-k}\prod_{1\leq i<k}\left(1-\frac{\theta}{\theta+n-i}\right)\,.
    \end{align*}
    Moreover, conditionally given that $|T_{n,\theta}(\lChild)|=k-1$, $T_{n,\theta}(\lChild)$ is distributed as a random binary search tree of size $k-1$, $T_{n,\theta}(\rChild)$ is distributed as a record-biased tree with parameters $n-k$ and $\theta$, and $T_{n,\theta}(\lChild)$ and $T_{n,\theta}(\rChild)$ are independent of each other. Conversely, the preceding properties completely characterize record-biased binary search trees.
\end{prop}

The proof of Proposition~\ref{prop:inductiveTree} can be found in Section~\ref{subsec:generating}. Note that this proposition implies that
\begin{align*}
    \mathbb{P}\Big(\big|T_{n,\theta}(\lChild)\big|\geq k\Big)=\prod_{1\leq i\leq k}\left(1-\frac{\theta}{\theta+n-i}\right)\,,
\end{align*}
a fact that will be used multiple times in what follows. This result implies that record-biased trees can be generated inductively starting from the split between left and right subtree at the root. It also states that only the right subtree keeps the $\theta$-record-biased distribution, since the left subtree is a random binary search tree.

Since this work studies the binary search trees associated to record-biased permutations, it is important to understand the relation between the number of records of the permutation and the structure of the tree. Given a permutation $\sigma$, say that $\sigma(i)$ is a record of $\sigma$ if, for all $j<i$, $\sigma(j)<\sigma(i)$. With this definition, note that $\sigma(i)$ is a record of $\sigma$ if and only if the unique node $u\in\bst{\sigma}$ labelled $\sigma(i)$ lies on the rightmost path of $\bst{\sigma}$. Next, for any subtree $T$ of $T_\infty$, let
\begin{align*}
    \rec(T):=\big|\big\{k:\rChild^k\in T\big\}\big|=1+\max\big\{k:\rChild^k\in T\big\}-\min\big\{k:\rChild^k\in T\big\}\,;
\end{align*}
by the above remark, if $T=\bst{\sigma}$, then $\rec(T)=\rec(\sigma)$. From this definition, it is also useful to note that, for a subtree $T$ of $T_\infty$, we have $h(T)\geq\rec(T)-1$.

An important input to the proofs, which is a fairly straightforward consequence of~\cite[Theorem~3]{auger2016analysis}, is the following lemma, stating bounds on the asymptotic behaviour of the number of records of record-biased permutations and trees.

\begin{lemma}\label{lem:boundsRecords}
    Let $\epsilon>0$. Then, there exists $\delta=\delta(\epsilon)>0$ such that, for all $n\in\N$ and $\theta\in[0,\infty)$, for any $w_{n,\theta}$-distributed permutation $\sigma$, we have
    \begin{align*}
        \mathbb{P}\left(\left|\frac{\rec(\sigma)}{\mu(n,\theta)}-1\right|>\epsilon\right)=\mathbb{P}\left(\left|\frac{\rec(T_{n,\theta})}{\mu(n,\theta)}-1\right|>\epsilon\right)\leq e^{-\delta\mu(n,\theta)}\,,
    \end{align*}
    where $\mu(n,\theta)$ is defined as in \eqref{eq:mu}.
\end{lemma}

The proof of Lemma~\ref{lem:boundsRecords}, whose first equality simply follows from the definition of $\rec(T)$, can be found in Section~\ref{subsec:rightLength}. Combining this results with the asymptotic behaviour of $\mu$ gives tight bounds on the asymptotic behaviour of the number of records of $T_{n,\theta}$. In particular, if $\theta$ is fixed, and as $n$ tends to infinity, we have $\rec(T_{n,\theta})=(\theta+o_\mathbb{P}(1))\log n$.

An easy way to see the interest in bounding the number of records of record-biased trees is apparent when considering the following lower bound for their height:
\begin{align*}
    h(T_{n,\theta})\geq\max\Big\{1+h\big(T_{n,\theta}(\lChild)\big),\,\rec(T_{n,\theta})-1\Big\}\,.
\end{align*}
Indeed, this bound, combined with Lemma~\ref{lem:boundsRecords}, implies that $h(T_{n,\theta})\geq(\theta+o_\mathbb{P}(1))\log n$. Moreover, from Proposition~\ref{prop:inductiveTree}, we can verify that
\begin{align*}
    \log\big|T_{n,\theta}(\lChild)\big|=\big(1+o_\mathbb{P}(1)\big)\log n\,.
\end{align*}
Since, conditioned on its size, $T_{n,\theta}(\lChild)$ is a random binary search tree, the case $\theta=1$ of Theorem~\ref{thm:height} (which was already established in~\cite{devroye1986note}) implies that
\begin{align*}
    h\big(T_{n,\theta}(\lChild)\big)=\big(c^*+o_\mathbb{P}(1)\big)\log n\,.
\end{align*}
These two results together show that
\begin{align*}
    h(T_{n,\theta})\geq\max\big\{c^*,\,\theta\big\}\log n+o_\mathbb{P}(\log n)\,,
\end{align*}
corresponding to the lower bound of Theorem~\ref{thm:height}.

For the upper bound, since the left subtrees in $T_{n,\theta}$ are all binary search trees, for which the heights are already known to be concentrated around their expected height~\cite{devroye1986note,drmota2003analytic,reed2003height}, we expect the height of the left subtrees to be well-behaved, conditioned on their sizes. This means that, in order to bound the height of record-biased trees from above, we essentially need to understand two quantities: the size of the rightmost path, and the sizes of the left subtrees hanging from that path.

Let $T_{n,\theta}$ be a record-biased tree with parameters $n$ and $\theta$. Since Lemma~\ref{lem:boundsRecords} already gives us strong bounds on the length of the rightmost path of $T_{n,\theta}$, it only remains to better understand the properties of the sizes of the left subtrees hanging from that path $(|T_{n,\theta}(\rChild^j\lChild)|)_{j\geq0}$. Using Proposition~\ref{prop:inductiveTree}, it is easy to verify that
\begin{align*}
    \mathbb{P}\Big(\big|T_{n,\theta}(\rChild^j\lChild)\big|\geq k\,\Big|\,\big|T_{n,\theta}(\lChild)\big|,\ldots,\big|T_{n,\theta}(\rChild^{j-1}\lChild)\big|\Big)=\prod_{1\leq i\leq k}\left(1-\frac{\theta}{\theta+n-\sum_{0\leq\ell<j}\left(|T_{n,\theta}(\rChild^\ell\lChild)|+1\right)-i}\right)\,;
\end{align*}
this equality coming from the fact that, given $|T_{n,\theta}(\lChild)|,\ldots,|T_{n,\theta}(\rChild^{j-1}\lChild)|$, the tree $T_{n,\theta}(\rChild^j)$ is a record-biased tree of size $n-\sum_{0\leq\ell<j}(|T_{n,\theta}(\rChild^\ell\lChild)|+1)$.

Given two random variables $X$ and $Y$, we say that $X$ is stochastically smaller than $Y$, and write $X\preceq Y$, if and only if, for all $t\in\mathbb{R}$, we have $\mathbb{P}(X\geq t)\leq\mathbb{P}(Y\geq t)$. Using the above mentioned properties, we can prove that the sizes of the left subtrees $(|T_{n,\theta}(\rChild^j\lChild)|)_{j\geq0}$ are stochastically bounded as follows.

\begin{prop}\label{prop:upperBoundL}
    Let $\theta\in[0,\infty)$ and $(B_j)_{j\geq0}$ be a sequence of independent and identically distributed $\mathrm{Beta}(\theta+1,1)$ random variables. Then, for any $n\in\N$ and $j\in\N$, we have
    \begin{align*}
        \big|T_{n,\theta}(\rChild^j\lChild)\big|\preceq j+n\prod_{0\leq i<j}B_i
    \end{align*}
\end{prop}

The proof of Proposition~\ref{prop:upperBoundL} can be found in Section~\ref{subsec:boundsLeft}. Using the bound from Proposition~\ref{prop:upperBoundL}, by the law of large numbers and the fact that $\mathbb{E}[\log B_0]=-\frac{1}{\theta}$, we expect to have
\begin{align*}
    \big|T_{n,\theta}(\rChild^j\lChild)\big|\leq j+ne^{-(1+o_\mathbb{P}(1))\frac{j}{\theta}}\,.
\end{align*}
Moreover, since $|T_{n,\theta}(\rChild^j\lChild)|=0$ whenever $j>\rec(T_{n,\theta})$ and, by Lemma~\ref{lem:boundsRecords}, $\rec(T_{n,\theta})=(\theta+o_\mathbb{P}(1))\log n$, the $j$ term in the upper bound will not notably contribute and the previous inequality can be strengthened as follows:
\begin{align*}
    \big|T_{n,\theta}(\rChild^j\lChild)\big|\leq ne^{-(1+o_\mathbb{P}(1))\frac{j}{\theta}}\,.
\end{align*}

By the definition of the height of trees, for any subtree $T$ of $T_\infty$, we have that
\begin{align*}
    h(T)=\max_{0\leq j\leq\rec(T)}\Big\{j+h\big(T(\rChild^j\lChild)\big)\Big\}\,.
\end{align*}
Replacing $T$ with $T_{n,\theta}$ in this equality and using the previous upper bound on the size of the left subtrees and the fact that, conditioned on their sizes, the left subtrees are independent random binary search trees, we obtain that
\begin{align*}
    h(T_{n,\theta})&\leq\max_{0\leq j\leq\rec(T_{n,\theta})}\left\{j+c^*\left(\log n-\big(1+o_\mathbb{P}(1)\big)\frac{j}{\theta}\right)\right\}\\
    &=\max_{0\leq j\leq\rec(T_{n,\theta})}\left\{c^*\log n+j\left(1-\frac{c^*}{\theta}+o_\mathbb{P}(1)\right)\right\}\,.
\end{align*}
The $o_\mathbb{P}(1)$ term comes from the law of large numbers for the heights of binary search trees and these heights are sufficiently concentrated that this term can in fact be taken out of the maximum. This leads to the following upper bound for the height of record-biased trees
\begin{align*}
    h(T_{n,\theta})\leq c^*\log n+\max\left\{0,\,1-\frac{c^*}{\theta}+o_\mathbb{P}(1)\right\}\rec(T_{n,\theta})\,.
\end{align*}
Finally, by using that $\rec(T_{n,\theta})=(\theta+o_\mathbb{P}(1))\log n$ from Lemma~\ref{lem:boundsRecords}, we obtain that
\begin{align*}
    h(T_{n,\theta})\leq\max\big\{c^*,\,\theta\big\}\log n+o_\mathbb{P}(\log n)\,,
\end{align*}
which corresponds to the upper bound of Theorem~\ref{thm:height}.

The rest of the paper is organized in two sections. The first one, Section~\ref{sec:RBTrees}, uses a generative model for record-biased permutations introduced in~\cite{auger2016analysis} to deduce properties of record-biased trees. The second one, Section~\ref{sec:height}, combines the previously described properties to prove the theorems.

Before moving into the details of the proof, we define an important set of events for the study of record-biased trees. Given any $\N$-valued sequence $K=(k_j)_{j\geq0}$, let $E_{n,\theta}(K)$ be the event that the left subtree sizes $(|T_{n,\theta}(\rChild^j\lChild)|)_{j\geq0}$ are given by the entries of $K$:
\begin{align}
    E_{n,\theta}(K):=\Big\{\big|T_{n,\theta}(\rChild^j\lChild)\big|=k_j,\forall j\in\N\Big\}\,.\label{eq:EK}
\end{align}
Naturally, we will only be interested in $K$ such that $\mathbb{P}(E_{n,\theta}(K))>0$. Note that, for any finite subtree $T$ of $T_\infty$, $\rec(T)$ corresponds to the unique value $r\geq1$ such that
\begin{align*}
    r+\sum_{0\leq j<r}\big|T(\rChild^j\lChild)\big|=|T|\,.
\end{align*}
This implies that, for any vector $K$ with $\mathbb{P}(E_{n,\theta}(K))>0$, there is a unique non-negative integer $r=r(K)$ such that $r+\sum_{0\leq j<r}k_j=n$ and on the event $E_{n,\theta}(K)$, necessarily $\rec(T_{n,\theta})=r(K)$. In particular, conditioning on $E_{n,\theta}(K)$ determines $\rec(T_{n,\theta})$.

\subsection{Related work}

As mentioned before, record-biased permutations were just recently introduced~\cite{auger2016analysis}; and prior to the current work, to the best of our knowledge, \cite{auger2016analysis} was the only paper studying the model. However, it is worth situating this model in a slightly larger literature on random permutations. An observation on this model made by the authors is that record-biased permutations can be bijectively mapped to Ewens permutations~\cite{ewens1972sampling} using the Foata bijection~\cite{foata1968netto}. This connection leads to interesting properties of both record-biased permutations and Ewens permutations related to their numbers of records and to their cycle structures. Moreover, although the record-biased model was originally connected with Ewens permutations~\cite{ewens1972sampling}, it is worth noting its strong resemblance to the Mallows model of permutations~\cite{mallows1957non}; this resemblance was part of the inspiration for the current work.

Contrary to record-biased permutations, the literature on binary search trees and their height is vast and we only provide a glimpse here. The first order asymptotic behaviour of the height of random binary search trees was proven by Devroye~\cite{devroye1986note}, who showed that their height is of order $(c^*+o_\mathbb{P}(1))\log n$; this result was built off a previous work of Pittel~\cite{pittel1984growing}, stating that the height was of order $(\alpha+o_\mathbb{P}(1))\log n$, while not being able to identify the constant $\alpha$. Since then, these results have been extended to higher order asymptotic behaviour~\cite{drmota2003analytic,reed2003height}, and to other models of increasing trees~\cite{broutin2008height,drmota2009height}. Recently, similar results upon which this work is based on, proved the first order asymptotic behaviour of the height of Mallows trees, as well as distributional limits under some assumptions on the parameters of the model~\cite{addario2021height}.

Finally, it is worth mentioning that the heights of binary search trees are often closely related to properties of extremes in \textit{branching random walks}. An important example comes from the results of~\cite{broutin2008height,devroye1986note}, which strongly rely on the Hammersley-Kingman-Biggins theorem~\cite{biggins1976first,hammersley1974postulates,kingman1975first}, providing a law of large numbers for the minimum of a branching random walk. Related results on the minimal position in branching random walks can be found in~\cite{addario2009minima,aidekon2013convergence,dekking1991limit,hu2009minimal}.

\section{Properties of record-biased trees}\label{sec:RBTrees}

In this section we provide useful properties of record-biased trees and prove Proposition~\ref{prop:inductiveTree}, Lemma~\ref{lem:boundsRecords} and Proposition~\ref{prop:upperBoundL}. Moreover, we use Proposition~\ref{prop:upperBoundL} to obtain upper tail bounds on the sizes of the left subtrees $(|T_{n,\theta}(\rChild^j\lChild)|)_{j\geq0}$. Most results follow from the generative model of record-biased permutations developed in~\cite{auger2016analysis}.

\subsection{Generating record-biased trees}\label{subsec:generating}

The following proposition was proven in~\cite{auger2016analysis} and gives an easy way to understand and generate record-biased permutations.

\begin{prop}\label{prop:permDistribution}
    Let $n\in\N$ and $\theta\in[0,\infty)$. Let $\sigma$ be a $w_{n,\theta}$-distributed permutation. Then, for all $k\in[n]$, we have
    \begin{align*}
        \mathbb{P}\big(\sigma^{-1}(1)=k\big)=\left\{\begin{array}{ll}
            \frac{\theta}{\theta+n-1} & \textrm{if $k=1$}\\
            \frac{1}{\theta+n-1} & \textrm{otherwise}
        \end{array}\right.\,.
    \end{align*}
    Moreover, by defining $A_i=[n]\setminus\{\sigma^{-1}(1),\ldots,\sigma^{-1}(i-1)\}$, we have
    \begin{align*}
        \mathbb{P}\big(\sigma^{-1}(i)=k\mid\sigma^{-1}(1),\ldots,\sigma^{-1}(i-1)\big)=\left\{\begin{array}{ll}
            \frac{\theta}{\theta+n-i} & \textrm{if $k=\min(A_i)$} \\
            \frac{1}{\theta+n-i} & \textrm{otherwise}
        \end{array}\right.\,.
    \end{align*}
\end{prop}

The previous proposition fully describes the joint distribution of the values of $\sigma^{-1}(1),\ldots,\sigma^{-1}(n)$, so uniquely characterizes record-biased permutations. Note that $\sigma(k)$ is a record if and only if $k=\min(A_i)$, and that the probability for $\sigma(k)$ to be a record is independent of the previous values of $\sigma^{-1}(1),\ldots,\sigma^{-1}(i-1)$, which gives a useful way to compute the number of records of a record-biased permutation, as stated in the following corollary.

\begin{cor}\label{cor:mgfRecord}
    Let $n\in\N$ and $\theta\in[0,\infty)$. Let $\sigma$ be a $w_{n,\theta}$-distributed permutation. Then, for any $t\in\mathbb{R}$, the moment generating function of $\rec(\sigma)$ is given by
    \begin{align*}
        \mathbb{E}\left[e^{t\cdot\rec(\sigma)}\right]=\prod_{1\leq i\leq n}\left(1+(e^t-1)\frac{\theta}{\theta+n-i}\right)\,.
    \end{align*}
\end{cor}

Given a permutation $\sigma\in\Sn$ with $k=\sigma(1)$, write $\undersigma$ for the unique permutation of $[k-1]$ corresponding to the ordering of $\sigma$ on the set $\sigma^{-1}([k-1])$, that is
\begin{align*}
    \undersigma^{-1}(i)<\undersigma^{-1}(j)\,\Longleftrightarrow\sigma^{-1}(i)<\sigma^{-1}(j)
\end{align*}
Similarly, write $\oversigma$ for the unique permutation of $[n-k]$ corresponding to the ordering of $\sigma$ on the set $\sigma^{-1}([n]\setminus[k])$, that is
\begin{align*}
    \oversigma^{-1}(i)<\oversigma^{-1}(j)\,\Longleftrightarrow\sigma^{-1}(k+i)<\sigma^{-1}(k+j)
\end{align*}
The following corollary gives an interesting characterization of $\sigma(1)$, $\undersigma$, and $\oversigma$ when $\sigma$ is a record-biased permutation. 

\begin{cor}\label{cor:inductivePermutation}
    Let $n\in\N$ and $\theta\in[0,\infty)$. Let $\sigma$ be a $w_{n,\theta}$-distributed permutation. Then, for any $k\in[n]$, we have
    \begin{align*}
        \mathbb{P}\big(\sigma(1)=k\big)=\frac{\theta}{\theta+n-k}\prod_{1\leq i<k}\left(1-\frac{\theta}{\theta+n-i}\right)\,.
    \end{align*}
    Moreover, given that $\sigma(1)=k$, $\undersigma$ is a uniformly random permutation of $\mathcal{S}_{k-1}$, $\oversigma$ is a record-biased permutation with parameters $n-k$ and $\theta$, and $\undersigma$ and $\oversigma$ are independent of each other.
\end{cor}

\begin{proof}
    For the distribution of $\sigma(1)$, note that we have
    \begin{align*}
        \mathbb{P}\big(\sigma(1)=k\mid\sigma(1)>k-1\big)=\mathbb{P}\big(\sigma^{-1}(k)=1\mid\sigma(1)>k-1\big)=\frac{\theta}{\theta+n-k}\,,
    \end{align*}
    where the second equality follows from Proposition~\ref{prop:permDistribution} and the fact that $1\in A_k\Leftrightarrow1=\min(A_k)$. By induction, this proves the desired distribution for $\sigma(1)$.
    
    For the distribution of $\undersigma$, note that, if $\sigma(1)=k$, then $1\notin\{\sigma^{-1}(1),\ldots,\sigma^{-1}(k-1)\}$. Combining this with Proposition~\ref{prop:permDistribution}, we have that, for any $i<k$, for any $j_1,\ldots,j_{i-1}$ all distinct in $[n]\setminus\{1\}$, and for any $j\in[n]\setminus\{1,j_1,\ldots,j_{i-1}\}$
    \begin{align*}
        &\mathbb{P}\Big(\sigma^{-1}(i)=j\,\Big|\,\sigma(1)=k,\sigma^{-1}(1)=j_1,\ldots,\sigma^{-1}(i-1)=j_{i-1}\Big)\\
        &\hspace{0.5cm}=\mathbb{P}\Big(\sigma^{-1}(i)=j\,\Big|\,\sigma^{-1}(i)\neq1,\sigma^{-1}(1)=j_1,\ldots,\sigma^{-1}(i-1)=j_{i-1}\Big)\\
        &\hspace{0.5cm}\propto\frac{1}{\theta+n-i}\,.
    \end{align*}
    Since the last value does not depend on $j$, $\sigma^{-1}(i)$ is uniformly distributed over $j\in[n]\setminus\{1,j_1,\ldots,j_{i-1}\}$. Moreover, the definition of $\undersigma$ implies that, for any $\tau\in\mathcal{S}_{k-1}$, we have
    \begin{align*}
        \mathbb{P}\Big(\undersigma=\tau\,\Big|\,\sigma(1)=k\Big)=\mathbb{P}\Big(\sigma^{-1}\big(\tau(1)\big)<\sigma^{-1}\big(\tau(2)\big)<\ldots<\sigma^{-1}\big(\tau(k-1)\big)\,\Big|\,\sigma(1)=k\Big)=\frac{1}{(k-1)!}
    \end{align*}
    which proves that, conditionally given $\sigma(1)=k$, the random permutation $\undersigma$ is uniformly distributed. Finally, for the distribution of $\oversigma$, by writing $A_{k+1}=[n]\setminus\{\sigma^{-1}(1),\ldots,\sigma^{-1}(k)\}=\{x_1<\ldots<x_{n-k}\}$, Proposition~\ref{prop:permDistribution} implies that, for any $\tau\in\mathcal{S}_{n-k}$, we have
    \begin{align*}
        &\mathbb{P}\Big(\sigma^{-1}(k+1)=x_{\tau(1)},\ldots,\sigma^{-1}(n)=x_{\tau(n-k)}\,\Big|\,\sigma^{-1}(1),\ldots,\sigma^{-1}(k-1),\sigma(1)=k\Big)\\
        &\hspace{0.5cm}=\prod_{k<i\leq n}\mathbb{P}\Big(\sigma^{-1}(i)=x_{\tau(i-k)}\,\Big|\,\sigma^{-1}(1),\ldots,\sigma^{-1}(k-1),\sigma(1)=k,\sigma^{-1}(j)=x_{\tau(j-k)},\forall k<j<i\Big)\\
        &\hspace{0.5cm}=w_{n-k,\theta}(\tau)\,.
    \end{align*}
    By the definition of $\oversigma$, this proves that
    \begin{align*}
        \mathbb{P}\Big(\oversigma=\tau\,\Big|\,\sigma^{-1}(1),\ldots,\sigma^{-1}(k-1),\sigma(1)=k\Big)=w_{n-k,\theta}(\tau)\,,    
    \end{align*}
    implying that, conditionally given $\sigma(1)=k$, the random permutation $\oversigma$ is $w_{n-k,\theta}$-distributed and independent of $\undersigma$.
\end{proof}

Proposition~\ref{prop:inductiveTree} now almost directly follows from Corollary~\ref{cor:inductivePermutation}.

\begin{proof}[Proof of Proposition~\ref{prop:inductiveTree}]
    The definition of binary search trees implies that, for any permutation $\sigma$, we have
    \begin{align*}
        \bst{\sigma}(\lChild)=\lChild\bst{\undersigma}
    \end{align*}
    and
    \begin{align*}
        \bst{\sigma}(\rChild)=\rChild\bst{\oversigma}\,.
    \end{align*}
    Combining this with Corollary~\ref{cor:inductivePermutation} proves that, conditionally given their sizes, $T_{n,\theta}(\lChild)$ is a random binary search tree, $T_{n,\theta}(\rChild)$ is a $\theta$-record-biased tree, and they are independent of each other. The direct statement of Proposition~\ref{prop:inductiveTree} now simply follows from Corollary~\ref{cor:inductivePermutation} and the fact that $|\bst{\sigma}(\lChild)|=\sigma(1)-1$. For the converse, use that binary search trees are completely characterized by their left and right subtree distributions to see that these distributions completely characterize record-biased trees.
\end{proof}

\subsection{Length of the rightmost path}\label{subsec:rightLength}

As explained in Section~\ref{subsec:overview}, to prove our results on the height of record-biased trees, we control the length of the rightmost path and bound the sizes of the left subtrees hanging from that path. We now prove Lemma~\ref{lem:boundsRecords}, bounding the number of records of a record-biased permutation, hence bounding the length of the rightmost path of the corresponding binary search tree.

\begin{proof}[Proof of Lemma~\ref{lem:boundsRecords}]
    Using Corollary~\ref{cor:mgfRecord}, for a $w_{n,\theta}$-distributed permutation $\sigma$, for any $t\in\mathbb{R}$, we have
    \begin{align*}
        \mathbb{P}\left[e^{t\cdot\rec(\sigma)}\right]=\prod_{1\leq i\leq n}\left(1+(e^t-1)\frac{\theta}{\theta+n-i}\right)\,.
    \end{align*}
    Now, using Chernoff's bound, it follows that
    \begin{align*}
        \mathbb{P}\Big(\rec(\sigma)>(1+\epsilon)\mu(n,\theta)\Big)\leq e^{-t(1+\epsilon)\mu(n,\theta)}\prod_{1\leq i\leq n}\left(1+(e^t-1)\frac{\theta}{\theta+n-i}\right)\,.
    \end{align*}
    Since $1+x\leq e^x$, we have that
    \begin{align*}
        \mathbb{P}\Big(\rec(\sigma)>(1+\epsilon)\mu(n,\theta)\Big)\leq\exp\Big(-t(1+\epsilon)\mu(n,\theta)+(e^t-1)\mu(n,\theta)\Big)\,,
    \end{align*}
    where $\mu(n,\theta)$ is defined in \eqref{eq:mu}. For $t$ small enough, $-t(1+\epsilon)+(e^t-1)<0$, proving the upper bound of the lemma. Similarly for the lower bound,
    \begin{align*}
        \mathbb{P}\Big(\rec(\sigma)<(1-\epsilon)\mu(n,\theta)\Big)&\leq e^{t(1-\epsilon)\mu(n,\theta)}\prod_{1\leq i\leq n}\left(1+(e^{-t}-1)\frac{\theta}{\theta+n-i}\right)\\
        &\leq\exp\Big(t(1-\epsilon)\mu(n,\theta)+(e^{-t}-1)\mu(n,\theta)\Big)\,,
    \end{align*}
    and once again, for $t$ small enough, $t(1-\epsilon)+(e^{-t}-1)<0$.
\end{proof}

\subsection{Stochastic bound on the left subtrees}\label{subsec:boundsLeft}

We conclude this section with results on the sizes of the left subtree of $T_{n,\theta}$. We start with a lemma useful to bound the size of the right subtree at the root of a record-biased tree.

\begin{lemma}\label{lem:boundFirstLeft}
    Let $\theta\in[0,\infty)$ and let $B$ be a $\mathrm{Beta}(\theta+1,1)$-distributed random variable. Then, for any $n\in\N$,
    \begin{align*}
    \big|T_{n,\theta}(\rChild)\big|\preceq nB+1\,.
    \end{align*}
\end{lemma}

\begin{proof}
    This statement holds when $n=0$ since $0\leq1$, so we can now assume that $n\geq1$. By Proposition~\ref{prop:inductiveTree}, for any $0\leq k\leq n-1$, we have
    \begin{align*}
        \mathbb{P}\Big(\big|T_{n,\theta}(\rChild)\big|\leq n-k-1\Big)=\mathbb{P}\Big(\big|T_{n,\theta}(\lChild)\big|\geq k\Big)=\prod_{1\leq i\leq k}\left(1-\frac{\theta}{\theta+n-i}\right)\,.
    \end{align*}
    By using that $\frac{1}{1+x}\geq e^{-x}$ along with the fact that $1-\frac{\theta}{\theta+n-i}=\frac{1}{1+\frac{\theta}{n-i}}$, this yields the lower bound
    \begin{align*}
        \mathbb{P}\Big(\big|T_{n,\theta}(\rChild)\big|\leq n-k-1\Big)\geq\exp\left(-\theta\sum_{1\leq i\leq k}\frac{1}{n-i}\right)\,.
    \end{align*}
    By comparison to the integral, it follows that
    \begin{align*}
        \mathbb{P}\Big(\big|T_{n,\theta}(\rChild)\big|\leq n-k-1\Big)\geq\exp\left(-\theta\int_{n-k-1}^{n-1}\frac{1}{t}dt\right)\geq\left(\frac{n-k-1}{n}\right)^\theta\,.
    \end{align*}
    Using this bound, for any $t\in[0,n-1]$, we have
    \begin{align*}
        \mathbb{P}\Big(\big|T_{n,\theta}(\rChild)\big|\leq t\Big)=\mathbb{P}\Big(\big|T_{n,\theta}(\rChild)\big|\leq \lfloor t\rfloor\Big)\geq\left(\frac{\lfloor t\rfloor}{n}\right)^\theta\,.
    \end{align*}
    Finally, use that
    \begin{align*}
        \mathbb{P}\big(nB+1\leq t\big)=\left(\left(\frac{t-1}{n}\right)_+\right)^\theta\leq\left(\frac{\lfloor t\rfloor}{n}\right)^\theta
    \end{align*}
    to conclude the proof of the proposition.
\end{proof}

Note that there is a natural stochastic lower bound for $|T_{n,\theta}(\rChild)|$, given by $nB-1-\theta$, whose proof essentially follows the same line of argument as the upper bound. However, this lower bound is not useful for our argument, so we omit it.

Using Lemma~\ref{lem:boundFirstLeft}, we can now prove Proposition~\ref{prop:upperBoundL}.

\begin{proof}[Proof of Proposition~\ref{prop:upperBoundL}]
    By Lemma~\ref{lem:boundFirstLeft},
    \begin{align*}
        \mathbb{P}\big(\big|T_{n,\theta}(\rChild)\big|\leq k\big)\geq\mathbb{P}\big(1+nB_0\leq k\big)\,.
    \end{align*}
    Now, by using properties of subtrees of a record-biased tree from Proposition~\ref{prop:inductiveTree}, we have that
    \begin{align*}
        \mathbb{P}\Big(\big|T_{n,\theta}(\rChild^j)\big|\leq k\,\Big|\,\big|T_{n,\theta}(\rChild^{j-1})\big|=\ell\Big)=\mathbb{P}\Big(\big|T_{\ell,\theta}(\rChild)\big|\leq k\Big)\geq\mathbb{P}\Big(1+\ell B_{j-1}\leq k\Big)\,.
    \end{align*}
    It follows that
    \begin{align*}
        \big|T_{n,\theta}(\rChild^j)\big|\preceq1+B_{j-1}\left(j-1+n\prod_{0\leq i<j-1}B_i\right)\leq j+n\prod_{0\leq i<j}B_i\,.
    \end{align*}
    To conclude this proof, simply use that $|T_{n,\theta}(\rChild^j\lChild)|\leq|T_{n,\theta}(\rChild^j)|$.
\end{proof}

We can now conclude this section with a rather sharp upper tail bound on the sizes of the left subtrees in $T_{n,\theta}$.

\begin{prop}\label{prop:boundL}
    Let $n\in\N$ and $\theta\in[0,\infty)$. Fix $\epsilon>0$ such that $\epsilon\theta<1$. Then, for any $M\in[0,\infty)$ and $k\in\N$ such that $ke^{\left(\frac{1}{\theta}-\epsilon\right)k}<ne^M$, we have
    \begin{align*}
        \mathbb{P}\Big(\exists j\leq k,\big|T_{n,\theta}(\rChild^j\lChild)\big|>ne^{-\left(\frac{1}{\theta}-\epsilon\right)+M}\Big)\leq Ce^{-\lambda M}\cdot\left(1-\frac{ke^{\left(\frac{1}{\theta}-\epsilon\right)k}}{ne^M}\right)^{-\lambda}
    \end{align*}
    where $C=\frac{1}{1-(1-\epsilon\theta)e^{\epsilon\theta}}$ and $\lambda=\frac{\epsilon\theta^2}{1-\epsilon\theta}$.
\end{prop}

\begin{proof}
    Fix $j\leq k$. Using the bound from Proposition~\ref{prop:upperBoundL}, we have that
    \begin{align*}
        \mathbb{P}\Big(\big|T_{n,\theta}(\rChild^j\lChild)\big|>ne^{-\left(\frac{1}{\theta}-\epsilon\right)+M}\Big)\leq\mathbb{P}\left(\prod_{0\leq i<j}B_i>e^{-\left(\frac{1}{\theta}-\epsilon\right)j+M}-\frac{j}{n}\right)\,.
    \end{align*}
    Now, let $c\in(0,\infty)$ and $t\in(0,\infty)$. Using Markov's inequality, we have
    \begin{align*}
        \mathbb{P}\left(\prod_{0\leq i<j}B_i>c\right)\leq\frac{1}{c^t}\mathbb{E}\left[\left(\prod_{0\leq i<j}B_i\right)^t\right]\leq\frac{1}{c^t}\mathbb{E}\left[B_0^t\right]^j=\frac{1}{c^t}\left(\frac{\theta}{\theta+t}\right)^j\,.
    \end{align*}
    Since $j\leq k$ and by assumption on the value of $k$, we know that $e^{-\left(\frac{1}{\theta}-\epsilon\right)j+M}-\frac{j}{n}\geq e^{-\left(\frac{1}{\theta}-\epsilon\right)k+M}-\frac{k}{n}>0$. This implies that we can apply the previous bound with $c=e^{-\left(\frac{1}{\theta}-\epsilon\right)j+M}-\frac{j}{n}$ and obtain
    \begin{align*}
        \mathbb{P}\Big(\big|T_{n,\theta}(\rChild^j\lChild)\big|>ne^{-\left(\frac{1}{\theta}-\epsilon\right)+M}\Big)&\leq\left(e^{-\left(\frac{1}{\theta}-\epsilon\right)j+M}-\frac{j}{n}\right)^{-t}\left(\frac{\theta}{\theta+t}\right)^j\\
        &=\exp\left(-t\log\left(e^{-\left(\frac{1}{\theta}-\epsilon\right)j+M}-\frac{j}{n}\right)+j\log\left(\frac{\theta}{\theta+t}\right)\right)\,.
    \end{align*}
    Write $\Xi_j=\frac{je^{\left(\frac{1}{\theta}-\epsilon\right)j}}{ne^M}<1$. Using that
    \begin{align*}
        \log\left(e^{-\left(\frac{1}{\theta}-\epsilon\right)j+M}-\frac{j}{n}\right)=-\left(\frac{1}{\theta}-\epsilon\right)j+M+\log(1-\Xi_j)
    \end{align*}
    and that $\Xi_j\leq\Xi_k$, we eventually have that
    \begin{align*}
        \mathbb{P}\Big(\big|T_{n,\theta}(\rChild^j\lChild)\big|>ne^{-\left(\frac{1}{\theta}-\epsilon\right)+M}\Big)\leq\exp\left(\left(\frac{1}{\theta}-\epsilon\right)tj-t\big(M+\log(1-\Xi_k)\big)+j\log\left(\frac{\theta}{\theta+t}\right)\right)\,.
    \end{align*}
    To conclude the proof, set $t=\lambda=\frac{\epsilon\theta^2}{1-\epsilon\theta}$ to obtain
    \begin{align*}
        \mathbb{P}\Big(\big|T_{n,\theta}(\rChild^j\lChild)\big|>ne^{-\left(\frac{1}{\theta}-\epsilon\right)+M}\Big)\leq\exp\bigg(\Big(\epsilon\theta+\log(1-\epsilon\theta)\Big)j-\lambda\Big(M+\log(1-\Xi_k)\Big)\bigg)\,.
    \end{align*}
    The desired result follows from using a union bound and summing the right hand side over $j\geq0$.
\end{proof}

\section{The height of record-biased trees}\label{sec:height}

In this section, the characteristics of record-biased trees developed in Section~\ref{sec:RBTrees} are used to obtain bounds on their height and eventually prove Theorem~\ref{thm:height} and \ref{thm:strongHeight}.

\subsection{Lower bounds on the height}

We start by proving the following result, corresponding to the lower bound of Theorem~\ref{thm:height}.

\begin{prop}\label{prop:lowerBound}
    Let $\theta\in[0,\infty)$ and $\epsilon>0$. Then, as $n$ tends to infinity
    \begin{align*}
        \mathbb{P}\Big(h(T_{n,\theta})<\big(\max\{c^*,\theta\}-\epsilon\big)\log n\Big)\longrightarrow0\,.
    \end{align*}
\end{prop}

\begin{proof}
    Using Lemma~\ref{lem:boundsRecords} and the fact that $h(T_{n,\theta})\geq\rec(T_{n,\theta})-1$ along with the asymptotic behaviour of $\mu(n,\theta)$ from \eqref{eq:mu} when $\theta$ is fixed, we first have that
    \begin{align*}
        \mathbb{P}\Big(h(T_{n,\theta})<\big(\theta-\epsilon\big)\log n\Big)\longrightarrow0\,,
    \end{align*}
    Note that, in the case $\theta=0$ the previous result trivially holds since $h(T_{n,\theta})>0$. It now only remains to prove that
    \begin{align*}
        \mathbb{P}\Big(h(T_{n,\theta})<\big(c^*-\epsilon\big)\log n\Big)\longrightarrow0\,.
    \end{align*}
    By the definition of the height of trees, we know that
    \begin{align*}
        h(T_{n,\theta})\geq1+h\big(T_{n,\theta}(\lChild)\big)\,.
    \end{align*}
    Hence, it suffices to show that
    \begin{align*}
        \mathbb{P}\Big(h\big(T_{n,\theta}(\lChild)\big)<\big(c^*-\epsilon\big)\log n\Big)\longrightarrow0\,.
    \end{align*}

    Using Proposition~\ref{prop:inductiveTree}, we know that
    \begin{align*}
        \mathbb{P}\Big(\big|T_{n,\theta}(\lChild)\big|\geq k\Big)=\prod_{1\leq i\leq k}\left(1-\frac{\theta}{\theta+n-i}\right)\,,
    \end{align*}
    from which it follows that
    \begin{align*}
        \mathbb{P}\left(\big|T_{n,\theta}(\lChild)\big|\geq\left\lfloor\frac{n}{\log n}\right\rfloor\right)&=\prod_{1\leq i\leq\frac{n}{\log n}}\left(1-\frac{\theta}{\theta+n-i}\right)\geq\left(1-\frac{\theta}{\theta+n}\right)^\frac{n}{\log n}=1-o(1)\,.
    \end{align*}
    Recall now from Proposition~\ref{prop:inductiveTree} that, given $|T_{n,\theta}(\lChild)|=k$, $T_{n,\theta}(\lChild)$ is distributed as a random binary search tree of size $k$. Using the previous lower bound on the size of $T_{n,\theta}(\lChild)$, and the law of large numbers for the height of a random binary search tree~\cite{devroye1986note}, it follows that
    \begin{align*}
        \mathbb{P}\Big(h\big(T_{n,\theta}(\lChild)\big)<(c^*-\epsilon)\log n\Big)\longrightarrow0\,.
    \end{align*}
    This establishes the second desired lower bound for the height of $T_{n,\theta}$ and concludes the proof.
\end{proof}

Proving the lower bound of Theorem~\ref{thm:strongHeight} is similar, but easier.

\begin{prop}\label{prop:strongLowerBound}
    Let $(\theta_n)_{n\geq0}$ be a sequence of non-negative numbers, such that $\theta_n\rightarrow\infty$. Then, for all $\epsilon>0$ and $\lambda>0$, we have
    \begin{align*}
        \mathbb{P}\Big(h(T_{n,\theta_n})<(1-\epsilon)\mu(n,\theta_n)\Big)=O\left(\frac{1}{n^\lambda}\right)\,.
    \end{align*}
\end{prop}

\begin{proof}
    Using Lemma~\ref{lem:boundsRecords} along with the fact that $h(T_{n,\theta})\geq\rec(T_{n,\theta})-1$, we have that
    \begin{align*}
        \mathbb{P}\Big(h(T_{n,\theta})\leq(1-\epsilon)\mu(n,\theta_n)\Big)\leq e^{-\delta(\epsilon)\mu(n,\theta_n)}\,.
    \end{align*}
    Moreover, since $\mu(n,\theta_n)=\omega(\log n)$ whenever $\theta_n\rightarrow\infty$, we know that $e^{-\delta(\epsilon)\mu(n,\theta_n)}=O(n^{-\lambda})$, which concludes the proof of the proposition.
\end{proof}

\subsection{Upper bounds on the height}

The first result of this section proves a very useful upper tail bound on the height of a record-biased tree, conditionally given the sizes of its left subtrees. Before stating this result, recall the definition of $E_{n,\theta}(K)$ when $K=(k_j)_{j\geq0}$ from \eqref{eq:EK}:
\begin{align*}
    E_{n,\theta}(K)=\Big\{\big|T_{n,\theta}(\rChild^j\lChild)\big|=k_j,\,\forall j\in\N\Big\}\,.
\end{align*}
Moreover, recall from the end of Section~\ref{subsec:overview} that, conditionally given $E_{n,\theta}(K)$ and given that $\mathbb{P}(E_{n,\theta}(K))>0$, there exists a unique $r=r(K)$ such that $\rec(T_{n,\theta})=r(K)$.

\begin{prop}\label{prop:upperBound}
    Let $n\in\N$ and $\theta\in[0,\infty)$. Consider $K=(k_j)_{j\geq0}$ such that $\mathbb{P}(E_{n,\theta}(K))>0$, where $E_{n,\theta}(K)$ is defined in \eqref{eq:EK} and let $r=r(K)$ be the unique value such that $r+\sum_{0\leq j<k}k_j=n$. Then, for any $\eta\in\N$ and $t>0$
    \begin{align*}
        \mathbb{P}\Big(h(T_{n,\theta})\geq\eta\,\Big|\,E_{n,\theta}(K)\Big)\leq\sum_{0\leq j\leq r}(2e^{-t})^{\eta-j}\cdot(k_j+1)^{e^t-1}
    \end{align*}
\end{prop}

\begin{proof}
    Thanks to Proposition~\ref{prop:inductiveTree}, for any vector $K=(k_j)_{j\geq0}$ such that $\mathbb{P}(E_{n,\theta}(K))>0$, conditionally given $E_{n,\theta}(K)$, the trees $(T_{n,\theta}(\rChild^j\lChild))_{j\geq0}$ are independent random binary search trees of respective sizes $K=(k_j)_{j\geq0}$. Using this fact together with the identity
    \begin{align*}
        h(T_{n,\theta})=\max_{0\leq j\leq\rec(T_{n,\theta})}\Big\{j+h\big(T_{n,\theta}(\rChild^j\lChild)\big)\Big\}\,.
    \end{align*}
    it follows that
    \begin{align*}
       \mathbb{P}\Big(h(T_{n,\theta})\geq\eta\,\Big|\,E_{n,\theta}(K)\Big)&=\mathbb{P}\left(\max_{0\leq j\leq\rec(T_{n,\theta})}\Big\{j+h\big(T_{n,\theta}(\rChild^j\lChild)\big)\Big\}\geq\eta\,\bigg|\,E_{n,\theta}(K)\right)\\
        &=\mathbb{P}\left(\max_{0\leq j\leq r}\Big\{j+h\big(\mathrm{RBST}_j\big)\Big\}\geq\eta\,\Big|\,E_{n,\theta}(K)\right)\,,
    \end{align*}
    where $(\mathrm{RBST}_j)_{j\geq0}$ are independent random binary search trees of respective sizes $k_j$, also independent of $E_{n,\theta}(K)$. Using a union bound, we obtain that
    \begin{align}
        \mathbb{P}\Big(h(T_{n,\theta})\geq\eta\,\Big|\,E_{n,\theta}(K)\Big)\leq\sum_{0\leq j\leq r}\mathbb{P}\Big(j+h\big(\mathrm{RBST}_j\big)\geq\eta\Big)\,.\label{eq:upperBound1}
    \end{align}
    
    To bound the right-hand side, use a union bound over the $2^{\eta-j}$ paths of length $\eta-j$ from the root of $T_\infty$ to obtain
    \begin{align}
        \mathbb{P}\Big(j+h\big(\mathrm{RBST}_j\big)\geq\eta\Big)\leq\sum_{v\in T_\infty:|v|=\eta-j}\mathbb{P}\big(v\in\mathrm{RBST}_j\big)=2^{\eta-j}\mathbb{P}\big(\rChild^{\eta-j}\in\mathrm{RBST}_j\big)\,.\label{eq:upperBound2}
    \end{align}
    Using the moment generating function from Corollary~\ref{cor:mgfRecord} and the fact that a binary search tree is simply a record-biased tree with parameter $\theta=1$, we have that
    \begin{align*}
        \mathbb{P}\big(\rChild^{\eta-j}\in\mathrm{RBST}_j\big)&=\mathbb{P}\Big(\rec(\mathrm{RBST}_j)\geq\eta-j+1\Big)\leq e^{-t(\eta-j+1)}\prod_{1\leq\ell\leq k_j}\left(1+(e^t-1)\frac{1}{k_j-\ell+1}\right)\,.
    \end{align*}
    Now, using that
    \begin{align*}
        \prod_{1\leq i\leq n}\left(1+(e^t-1)\frac{1}{n-i+1}\right)\leq\exp\left((e^t-1)\sum_{1\leq\ell\leq k_j}\frac{1}{\ell}\right)\leq\exp\Big((e^t-1)\log(k_j+1)\Big)
    \end{align*}
    and that $e^{-t}\leq 1$, we obtain that
    \begin{align}
        \mathbb{P}\big(\rChild^{\eta-j}\in\mathrm{RBST}_j\big)\leq\exp\Big(-t(\eta-j)+(e^t-1)\log(k_j+1)\Big)=e^{-t(\eta-j)}\cdot(k_j+1)^{e^t-1}\,.\label{eq:upperBound3}
    \end{align}
    Combining \eqref{eq:upperBound1}, \eqref{eq:upperBound2}, and \eqref{eq:upperBound3}, we obtain the desired bound.
\end{proof}

We now use Proposition~\ref{prop:upperBound} to prove the following result, which allows us to extend our convergence in probability results to the convergence in $L^p$. Recall the definition of $\mu(n,\theta)$ from \eqref{eq:mu}.

\begin{prop}\label{prop:UI}
    For any sequence $(\theta_n)_{n\geq0}$ and any $p>0$, the family of random variables
    \begin{align*}
        \left(\left(\frac{h(T_{n,\theta_n})}{\max\{c^*\log n,\mu(n,\theta_n)\}}\right)^p\right)_{n\geq0}
    \end{align*}
    is uniformly integrable.
\end{prop}

\begin{proof}
    First, write $\nu_n=\max\{c^*\log n,\,\mu(n,\theta_n)\}$ and fix $a>0$. We have that
    \begin{align*}
        \mathbb{P}\left(\left(\frac{h(T_{n,\theta_n})}{\nu_n}\right)^p\geq a\right)=\mathbb{P}\big(h(T_{n,\theta_n})\geq a^\frac{1}{p}\nu_n\big)\,.
    \end{align*}
    Using Proposition~\ref{prop:upperBound} with $\eta=\lceil a^\frac{1}{p}\nu_n\rceil$ and the fact that $k_j\leq n$, it follows that
    \begin{align*}
        \mathbb{P}\Big(h(T_{n,\theta_n})\geq\eta\,\Big|\,E_{n,\theta_n}(K)\Big)\leq\sum_{0\leq j\leq r(K)}(2e^{-t})^{\eta-j}(n+1)^{e^t-1}\,.
    \end{align*}
    Moreover, by using Lemma~\ref{lem:boundsRecords}, we know that there exists a universal constant $\delta>0$ such that
    \begin{align*}
        \mathbb{P}\Big(\rec(T_{n,\theta_n})>2\mu(n,\theta_n)\Big)\leq e^{-\delta\mu(n,\theta_n)}\,.
    \end{align*}
    In the case where $\mu(n,\theta_n)$ is bounded, the previous tail bound might not suffice for our purpose. However, in this case, we would have that $\mathbb{E}[\rec(T_{n,\theta_n})]=\mu(n,\theta_n)=O(1)$, from which a simple application of Markov's inequality gives us that
    \begin{align*}
        \mathbb{P}\Big(\rec(T_{n,\theta_n})>2\mu(n,\theta_n)+\log(n+1)\Big)=O\left(\frac{1}{\log n}\right)=o(1)\,.
    \end{align*}
    Combining both bounds, we have that
    \begin{align*}
        \mathbb{P}\Big(\rec(T_{n,\theta_n})>2\mu(n,\theta_n)+\log(n+1)\Big)=o(1)\,.
    \end{align*}
    
    Combine the previous tail bounds to obtain that
    \begin{align*}
        \mathbb{P}\Big(h(T_{n,\theta_n})\geq\eta\Big)&\leq\sum_{K:r(K)\leq 2\mu(n,\theta_n)+\log(n+1)}\mathbb{P}\Big(h(T_{n,\theta_n})\geq\eta\,\Big|\,E_{n,\theta_n}(K)\Big)\mathbb{P}\Big(E_{n,\theta_n}(K)\Big)+o(1)\\
        &\leq\sum_{0\leq j\leq2\mu(n,\theta_n)+\log(n+1)}(2e^{-t})^{\eta-j}(n+1)^{e^t-1}+o(1)\,.
    \end{align*}
    Let $t=\log(2e)$. The previous equation can then be rewritten as
    \begin{align*}
        \mathbb{P}\Big(h(T_{n,\theta_n})\geq\eta\Big)&\leq\sum_{0\leq j\leq2\mu(n,\theta_n)+\log(n+1)}\frac{1}{e^{\eta-j}}(n+1)^{1+e}+o(1)\\
        &=\frac{1}{e-1}\exp\Big(2\mu(n,\theta_n)+(2+e)\log(n+1)-\eta\Big)+o(1)\,.
    \end{align*}
    Recall now that $\eta=\lceil a^\frac{1}{p}\nu_n\rceil$ and that $\nu_n=\max\{c^*\log n,\,\mu(n,\theta_n)\}$. This implies that, whenever $a>\left(2+\frac{2+e}{c^*}\right)^p$, we have
    \begin{align*}
        \eta-2\mu(n,\theta_n)-(2+e)\log(n+1)>\left(a^\frac{1}{p}-2-\frac{2+e}{c^*}+o(1)\right)\max\big\{c^*\log n,\,\mu(n,\theta_n)\big\}\longrightarrow\infty\,.
    \end{align*}
    It follows that, for such $a$,
    \begin{align*}
        \lim_{n\rightarrow\infty}\mathbb{P}\left(\left(\frac{h(T_{n,\theta_n})}{\nu_n}\right)^p\geq a\right)=0\,,
    \end{align*}
    establishing the claimed uniform integrability.
\end{proof}

\subsection{Convergence of the height}

In this section, we build on the results of the previous sections to finally prove Theorem~\ref{thm:height} and \ref{thm:strongHeight}.

\begin{proof}[Proof of Theorem~\ref{thm:height}]
    Thanks to the uniform integrability proved in Proposition~\ref{prop:UI}, we only need to establish the convergence in probability in order to prove the theorem. Moreover, Proposition~\ref{prop:lowerBound} already proved the requisite lower bound, so it only remains to prove the upper bound.
    
    Assume first that $\theta=0$. In this case, $\rec(T_{n,\theta})=1$ and we have that
    \begin{align*}
        T_{n,\theta}=\{\varnothing\}\cup\big(\lChild T_{n,\theta}(\lChild)\big)\,,
    \end{align*}
    which implies that
    \begin{align*}
        h(T_{n,\theta})=1+h\big(T_{n,\theta}(\lChild)\big)\,.
    \end{align*}
    Using that, in this case, $T_{n,\theta}(\lChild)$ is random binary search tree of size $n-1$, along with the law of large numbers for the height of random binary search trees~\cite{devroye1986note}, it follows that
    \begin{align*}
        \mathbb{P}\Big(h(T_{n,\theta})>(c^*+\epsilon)\log n\Big)\longrightarrow0\,,
    \end{align*}
    which is the desired upper bound. Assume for the rest of the proof that $\theta>0$.
    
    For any $a>0$, by using the bound from Proposition~\ref{prop:upperBound}, we have
    \begin{align*}
        \mathbb{P}\Big(h(T_{n,\theta})\geq a\log n\,\Big|\,E_{n,\theta}(K)\Big)\leq\sum_{0\leq j\leq r(K)}(2e^{-t})^{\lfloor a\log n\rfloor-j}\cdot(k_j+1)^{e^t-1}\,.
    \end{align*}
    Moreover, from Lemma~\ref{lem:boundsRecords}, we know that, for any $\alpha>0$, as $n$ tends to infinity
    \begin{align*}
        \mathbb{P}\Big(\rec(T_{n,\theta})>(\theta+\alpha)\log n\Big)\longrightarrow0\,.
    \end{align*}
    Finally, Proposition~\ref{prop:boundL} tells us that, for any fixed $\beta>0$ such that $\beta\theta<1$, by setting $k=\left\lfloor\left(\frac{1}{\theta}-\beta\right)^{-1}\log n\right\rfloor$ and $M=2\log\log n$, we have
    \begin{align*}
        \mathbb{P}\Big(\exists 0\leq j\leq k,|T_{n,\theta}(\rChild^j\lChild)|>ne^{-\left(\frac{1}{\theta}-\beta\right)j+M}\Big)\leq Ce^{-\lambda M}\cdot\left(1-\frac{k}{e^M}\right)^{-\lambda}\longrightarrow0\,.
    \end{align*}
    Now fix $\alpha,\beta>0$ such that $\left(\frac{1}{\theta}-\beta\right)^{-1}>(\theta+\alpha)$, for example we could take $0<\beta<\frac{1}{\theta}$ and set $\alpha=\frac{\beta\theta}{2}\left(\frac{1}{\theta}-\beta\right)^{-1}$. Combining the previous results, we obtain
    \begin{align*}
        \mathbb{P}\Big(h(T_{n,\theta})\geq a\log n\Big)&\leq\sum_{K:\left\{\genfrac{}{}{0pt}{}{r(K)\leq(\theta+\alpha)\log n\hfill}{k_j\leq n\cdot\exp\left(-\left(\frac{1}{\theta}-\beta\right)j+M\right)}\right.}\mathbb{P}\Big(h(T_{n,\theta})\geq a\log n\,\Big|\,E_{n,\theta}(K)\Big)\mathbb{P}\Big(E_{n,\theta}(K)\Big)+o(1)\\
        &\leq\sum_{0\leq j\leq(\theta+\alpha)\log n}(2e^{-t})^{\lfloor a\log n\rfloor-j}\cdot\left(ne^{-\left(\frac{1}{\theta}-\beta\right)j+M}+1\right)^{e^t-1}+o(1)\,,
    \end{align*}
    where the bounds on $k_j$ hold since we assumed that $\left(\frac{1}{\theta}-\beta\right)^{-1}>(\theta+\alpha)$, implying that $j\leq(\theta+\alpha)\log n\leq k=\left\lfloor\left(\frac{1}{\theta}-\beta\right)^{-1}\log n\right\rfloor$ for $n$ large enough. Moreover, since $\left(\frac{1}{\theta}-\beta\right)j\leq\left(\frac{1}{\theta}-\beta\right)(\theta+\alpha)\log n<\log n$, it follows that $ne^{-\left(\frac{1}{\theta}-\beta\right)j+M}+1\leq e^{\log n-\left(\frac{1}{\theta}-\beta\right)j+o(\log n)}$. Similarly, $(2e^{-t})^{\lfloor a\log n\rfloor-j}=e^{(\log 2-t)(a\log n-j)+o(\log n)}$. Combining these formulas with the previous upper tail bound on the height, we obtain that, for any $0<\beta<\frac{1}{\theta}$, and for any $0<\alpha<\beta\theta\left(\frac{1}{\theta}-\beta\right)^{-1}$,
    \begin{align*}
        \mathbb{P}\Big(h(T_{n,\theta})\geq a\log n\Big)&\leq\sum_{0\leq j\leq(\theta+\alpha)\log n}e^{(\log 2-t)(a\log n-j)+o(\log n)}\cdot\left(e^{\log n-\left(\frac{1}{\theta}-\beta\right)j+o(\log n)}\right)^{e^t-1}+o(1)\\
        &=e^{[a(\log2-t)+(e^t-1)+o(1)]\log n}\sum_{0\leq j\leq(\theta+\alpha)\log n}\left(e^{-\log2+t-(e^t-1)\left(\frac{1}{\theta}-\beta\right)}\right)^j+o(1)\,.
    \end{align*}
    Set now $t=\log c^*$ and recall that $c^*\log\left(\frac{2e}{c^*}\right)=1$ to eventually rewrite the previous bound into
    \begin{align}
        &\mathbb{P}\Big(h(T_{n,\theta})\geq a\log n\Big)\label{eq:height1}\\
        &\hspace{0.5cm}\leq\exp\left(-(c^*-1)\left[\frac{a}{c^*}-1+o(1)\right]\log n\right)\sum_{0\leq j\leq(\theta+\alpha)\log n}\exp\left((c^*-1)\left[\frac{1}{c^*}-\frac{1}{\theta}+\beta\right]j\right)+o(1)\,.\notag
    \end{align}
    We now bound the right-hand side from above in two cases, according to whether $\theta<c^*$ or $\theta\geq c^*$.
    
    \textbf{Case 1:} $\theta<c^*$. In this case, for $\beta$ small enough, the sum in $j$ in \eqref{eq:height1} can be bounded by a constant depending only on $c^*$ and $\theta$. Moreover, for any $\epsilon>0$, taking $a=c^*+\epsilon$, the term before the sum converges to $0$. This proves that
    \begin{align*}
        \mathbb{P}\Big(h(T_{n,\theta})\geq (c^*+\epsilon)\log n\Big)\longrightarrow0\,.
    \end{align*}
    Since we assumed that $\theta<c^*$, this corresponds to the desired upper bound in this case.
    
    \textbf{Case 2:} $\theta\geq c^*$. For any $\epsilon>0$, taking $a=\theta+\epsilon$ and using \eqref{eq:height1}, we have that
    \begin{align}
        &\mathbb{P}\Big(h(T_{n,\theta})\geq(\theta+\epsilon)\log n\Big)\label{eq:height2}\\
        &\hspace{0cm}\leq\exp\left(-(c^*-1)\left[\frac{\theta+\epsilon}{c^*}-1+o(1)\right]\log n\right)\sum_{0\leq j\leq(\theta+\alpha)\log n}\exp\left((c^*-1)\left[\frac{1}{c^*}-\frac{1}{\theta}+\beta\right]j\right)+o(1)\notag\\
        &\hspace{0cm}\leq\exp\left(-(c^*-1)\left[\frac{\theta+\epsilon}{c^*}-1+o(1)\right]\log n\right)\exp\left(-(c^*-1)\left[\frac{1}{c^*}-\frac{1}{\theta}+\beta\right](\theta+\alpha)\log n+O(1)\right)+o(1)\notag\\
        &\hspace{0cm}=\exp\left(-(c^*-1)\left[\frac{\theta+\epsilon}{c^*}-1-\left(\frac{1}{c^*}-\frac{1}{\theta}+\beta\right)(\theta+\alpha)\right]\log n+o(\log n)\right)+o(1)\,,\notag
    \end{align}
    In order to conclude the proof in this case, note that
    \begin{align*}
        \frac{\theta+\epsilon}{c^*}-1-\left(\frac{1}{c^*}-\frac{1}{\theta}+\beta\right)(\theta+\alpha)&=\frac{\epsilon}{c^*}-\alpha\left(\frac{1}{c^*}-\frac{1}{\theta}\right)-\beta\theta-\alpha\beta\,.
    \end{align*}
    Moreover, the upper tail bound of \eqref{eq:height2} holds for any $\epsilon>0$, $0<\beta<\frac{1}{\theta}$, and $0<\alpha<\beta\theta\left(\frac{1}{\theta}-\beta\right)^{-1}$. Hence, for any fixed $\epsilon>0$, we can choose $\alpha$ and $\beta$ small enough, so that
    \begin{align*}
        \frac{\theta+\epsilon}{c^*}-1-\left(\frac{1}{c^*}-\frac{1}{\theta}+\beta\right)(\theta+\alpha)>0\,.
    \end{align*}
    This implies that
    \begin{align*}
        \mathbb{P}\Big(h(T_{n,\theta})\geq(\theta+\epsilon)\log n\Big)\longrightarrow0\,,
    \end{align*}
    which proves the desired upper bound in the case $\theta\geq c^*$. This concludes the proof of Theorem~\ref{thm:height}.
\end{proof}

We conclude this section with the proof of Theorem~\ref{thm:strongHeight}.

\begin{proof}[Proof of Theorem~\ref{thm:strongHeight}]
    Let $(\theta_n)_{n\geq0}$ be such that $\theta_n\rightarrow\infty$ and fix $\epsilon>0$. We start by proving the desired upper tail bound.
    
    Using Proposition~\ref{prop:upperBound}, we know that
    \begin{align*}
        \mathbb{P}\Big(h(T_{n,\theta_n})\geq\eta\,\Big|\,E_{n,\theta}(K)\Big)\leq\sum_{0\leq j\leq r(K)}(2e^{-t})^{\eta-j}\cdot(k_j+1)^{e^t-1}\,.
    \end{align*}
    Moreover, we know that $k_j\leq n$. Using this upper bound for $k_j$ and letting $t=\log(2e)$ and $\eta=\lceil(1+2\epsilon)\mu(n,\theta_n)\rceil$, we obtain
    \begin{align*}
        \mathbb{P}\Big(h(T_{n,\theta_n})\geq(1+2\epsilon)\mu(n,\theta_n)\,\Big|\,E_{n,\theta}(K)\Big)&\leq\sum_{0\leq j\leq r(K)}e^{j-\eta}\cdot(n+1)^{e+1}\\
        &\leq\frac{1}{e-1}\cdot(n+1)^{e+1}\cdot e^{r(K)-\eta}\,.
    \end{align*}
    Finally, using Lemma~\ref{lem:boundsRecords}, we know that
    \begin{align*}
        \mathbb{P}\Big(\rec(T_{n,\theta_n})>(1+\epsilon)\mu(n,\theta_n)\Big)\leq e^{-\delta(\epsilon)\mu(n,\theta_n)}\,,
    \end{align*}
    from which we obtain that
    \begin{align*}
        &\mathbb{P}\Big(h(T_{n,\theta_n})\geq(1+2\epsilon)\mu(n,\theta_n)\Big)\\
        &\hspace{0.5cm}\leq\sum_{K:r(K)\leq(1+\epsilon)\mu(n,\theta_n)}\mathbb{P}\Big(h(T_{n,\theta_n})\geq(1+2\epsilon)\mu(n,\theta_n)\,\Big|\,E_{n,\theta}(K)\Big)\mathbb{P}\Big(E_{n,\theta_n}(K)\Big)+e^{-\delta(\epsilon)\mu(n,\theta_n)}\\
        &\hspace{0.5cm}\leq\frac{1}{e-1}\cdot(n+1)^{e+1}\cdot e^{(1+\epsilon)\mu(n,\theta_n)-\eta}+e^{-\delta(\epsilon)\mu(n,\theta_n)}\,.
    \end{align*}
    From the definition of $\eta$, we know that $(1+\epsilon)\mu(n,\theta_n)-\eta\leq-\epsilon\mu(n,\theta_n)$. Moreover, in the case where $\theta_n\rightarrow\infty$, we know that $\mu(n,\theta_n)=\omega(\log n)$, from which we obtain that
    \begin{align*}
        \mathbb{P}\Big(h(T_{n,\theta_n})\geq(1+2\epsilon)\mu(n,\theta_n)\Big)\leq e^{-\epsilon\mu(n,\theta_n)+o(\mu(n,\theta_n))}+e^{-\delta(\epsilon)\mu(n,\theta_n)}=O\left(\frac{1}{n^\lambda}\right)\,.
    \end{align*}
    This corresponds to the desired upper tail bound. Combining this result with the lower tail bound from Proposition~\ref{prop:strongLowerBound}, we proved both the tail bound of the theorem and the convergence in probability. To conclude this proof, use Proposition~\ref{prop:UI} to extend the convergence in probability to the convergence in $L^p$ for any $p>0$.
\end{proof}

\section{Future work and open questions}

\begin{itemize}
    \item This work describes the first order asymptotic behaviour of the height of record-biased trees, and it is natural to next consider lower-order fluctuations. In the regime where $\theta>c^*$, the height is principally controlled by the number of records, and the latter quantity satisfies a central limit theorem; although we do not prove this, it is fairly straightforward to do so using the methods developed in this paper. It is is therefore natural to suspect that the height also satisfies a central limit theorem. On the other hand, when $\theta<c^*$, the height of record-biased permutations and the results of~\cite{drmota2003analytic,reed2003height} suggest that, in this case, $h(T_{n,\theta})$ has asymptotically bounded variance as $n\rightarrow\infty$. In the case that $\theta=c^*$, or more generally that $\theta_n\rightarrow c^*$, it is less clear to us what second-order behaviour to expect.
    \item The similarities between this work and~\cite{addario2021height} raise the question about more general frameworks under which these two models and results fall. Possible developments might come from using a more general model of random trees~\cite{evans2012trickle}, or a more general model of random permutations~\cite{pitman2019regenerative} upon which the binary search trees are built. Considering the height of these generalized models of trees might lead to a unifying framework for Mallows and record-biased trees, as well as other models.
\end{itemize}

\subsection*{Acknowledgements}

BC wishes to thank his supervisor, Louigi Addario-Berry, for his help with the general structure and presentation of this paper. During the preparation of this research, BC was supported by an ISM scholarship.

\medskip
\bibliographystyle{siam}
\bibliography{articles}

\end{document}